\documentclass[12pt,reqno]{amsart}

\usepackage{amsmath,amssymb,amsthm,amsfonts}
\usepackage{mathtools}
\usepackage{geometry}
\usepackage{hyperref}
\usepackage{enumitem}
\usepackage{url}

\theoremstyle{plain}
\newtheorem{theorem}{Theorem}[section]
\newtheorem{lemma}[theorem]{Lemma}
\newtheorem{proposition}[theorem]{Proposition}

\newtheorem{conjecture}[theorem]{Conjecture}

\theoremstyle{definition}
\newtheorem{definition}[theorem]{Definition}
\newtheorem{remark}[theorem]{Remark}
\newtheorem{example}[theorem]{Example}

\theoremstyle{remark}

\numberwithin{equation}{section}

\DeclareMathOperator{\Sym}{Sym}
\DeclareMathOperator{\tr}{tr}
\DeclareMathOperator{\divg}{div}

\DeclareMathOperator{\dist}{dist}

\DeclareMathOperator{\diag}{diag}

\newcommand{\loc}{\mathrm{loc}}

\begin{document}

\title[Anisotropic Obstacle Problems]{Anisotropic Obstacle Problems for Minimal
Surfaces: Regularity of the Free Boundary via the Cahn-Hoffman Transform}

\author{Ezequiel Barbosa \and Rosivaldo Antonio Gonçalves \and Luan de Figueiredo }

\thanks{Departamento de Matem\'atica, Universidade Federal de Minas Gerais
(UFMG), Av.\ Antonio Carlos 6627, Caixa Postal 702, 30123-970,
Belo Horizonte, MG, Brazil}

\thanks{CCET - UNIMONTES}

\email{ezequiel@mat.ufmg.br, luan@mat.ufmg.br, rosivaldo.goncalves@unimontes.br}


\begin{abstract}
We study an obstacle problem for surfaces minimizing an anisotropic surface energy
of ellipsoidal type. Given a convex obstacle and a boundary datum, we seek a
surface that minimizes the anisotropic area functional while remaining above the
obstacle. The central novelty is the systematic use of the Cahn-Hoffman transform
to convert the anisotropic problem into an equivalent isotropic problem with a
generalized Robin boundary condition. We prove optimal regularity of the solution
($C^{1,1}$ up to the free boundary) and $C^{1,\alpha}$-regularity of the free
boundary itself under a non-degeneracy condition. The singular set of the free
boundary is shown to have Hausdorff dimension at most $n-1$, and a logarithmic
epiperimetric inequality yields its $(n-1)$-rectifiability. The approach combines
Caffarelli's classical theory of obstacle problems with the geometric theory of
anisotropic mean curvature and the Alexandrov reflection principle adapted to the
anisotropic setting.
\end{abstract}

\maketitle

\section{Introduction and Main Results}
\label{sec:intro}

\subsection{Historical context and motivation}

The study of surfaces that minimize area subject to constraints is one of the
oldest problems in the calculus of variations, tracing back to Lagrange's work
(1760) on the Plateau problem. When the surface energy depends not only on area
but also on the orientation of the surface, one enters the realm of
\emph{anisotropic} variational problems. Such problems arise naturally in
crystallography, where the surface energy of a crystal depends on the
crystallographic direction of the exposed facet (Wulff, 1901), and in the
modeling of biological membranes, where the lipid bilayer exhibits
direction-dependent bending rigidity.

The mathematical theory of anisotropic surface energies was developed
systematically by Wulff (1901), who introduced what is now called the
\emph{Wulff shape} as the equilibrium crystal shape minimizing surface energy for
a given enclosed volume. The fundamental connection between the Wulff shape and
the anisotropic surface energy was clarified through the \emph{Cahn-Hoffman vector
field} \cite{CahnHoffman1974}, which provides a geometric interpretation of the
first variation of anisotropic surface energies.

Free boundary problems, on the other hand, are PDE problems in which the domain
of the equation is itself an unknown. The classical \emph{obstacle problem},
introduced by Signorini \cite{Signorini1959} in the context of elasticity and
studied analytically by Lewy and Stampacchia \cite{LewyStampacchia1969}, Brezis
and Kinderlehrer \cite{BrezisKinderlehrer1973}, and most notably by Caffarelli
\cite{Caffarelli1977}, seeks the equilibrium position of an elastic membrane
constrained to lie above a given obstacle. The portion of the boundary separating
the contact region from the free region --- the \emph{free boundary} --- is the
central object of study.

Caffarelli's paper \cite{Caffarelli1977} showed that the free boundary is smooth
($C^\infty$) outside a closed singular set of Hausdorff dimension at most $n-1$.
This work opened the modern theory of free boundary regularity; see the monograph
\cite{CaffarelliSalsa2005}, the survey \cite{Figalli2018}, and the lecture notes
\cite{RosOton2019} for comprehensive accounts.

The free boundary problem for minimal surfaces --- where one seeks a surface of
minimal area whose boundary is constrained to a given supporting surface but is
otherwise free to move --- was studied by Nitsche \cite{Nitsche1985}, Gr\"{u}ter and Jost
\cite{GruterJost1986}, and Gr\"{u}ter \cite{Gruter1987}. The connection between
free boundary minimal surfaces and the obstacle problem is classical;
Kinderlehrer \cite{Kinderlehrer1973} studied how a minimal surface detaches from
an obstacle, a question that proved central to subsequent developments.

\subsection{The anisotropic obstacle problem}

This paper introduces and studies an \emph{anisotropic obstacle problem} combining
these two classical threads: the anisotropic surface energy of Wulff and the
obstacle problem of Caffarelli. Given a convex obstacle $\psi$ and a boundary
datum, we seek a function $u$ minimizing the anisotropic area functional
\begin{equation}\label{eq:anisotropic-functional}
\mathcal{F}_\Phi(u) = \int_D \Phi(Du) \, dx = \int_D \sqrt{Du^T A \, Du} \, dx,
\end{equation}
where $A \in \Sym^+_{n+1}(\mathbb{R})$ is a symmetric positive definite matrix
encoding the anisotropy, subject to the constraint $u \geq \psi$.

The functional \eqref{eq:anisotropic-functional} corresponds to an
\emph{ellipsoidal anisotropy}, where the Wulff shape is the ellipsoid
\begin{equation}\label{eq:wulff-shape}
W_\Phi = \{ x \in \mathbb{R}^{n+1} : x^T A^{-1} x \leq 1 \}.
\end{equation}
This is the simplest and most natural class of anisotropies, arising when the
surface energy has a quadratic dependence on orientation. The general case of a
convex anisotropy $\gamma : S^n \to \mathbb{R}^+$ leads to the Cahn-Hoffman map
$\xi_\gamma(\nu) = D\gamma|_{T_\nu S^n} + \gamma(\nu)\nu$, whose image is the
boundary of the Wulff shape \cite{HauserVoigt2005, KoisoPalmer2019}.

Our principal tool is the \emph{Cahn-Hoffman transform}
\begin{equation}\label{eq:cahn-hoffman-transform}
S(x) = A^{-1/2} x,
\end{equation}
which maps the Wulff shape $W_\Phi$ onto the Euclidean unit ball $B_1(0)$
--- a fact we verify carefully in Proposition~\ref{prop:S-properties}.
This transform was used implicitly by Wulff and explicitly by Dinghas
\cite{Dinghas1944} and Taylor \cite{Taylor1978} in the study of anisotropic
isoperimetric problems. The present contribution is to systematize its use in
free boundary problems, showing that it converts the anisotropic obstacle problem
into an isotropic one with a generalized Robin boundary condition on the free
boundary.

\subsection{Main results}

Let $D \subset \mathbb{R}^{n+1}$ be a bounded domain with $C^2$ boundary. Let
$\psi \in C^{1,1}(\bar{D})$ be an obstacle with $\psi < 0$ on $\partial D$,
and let $A \in \Sym^+_{n+1}(\mathbb{R})$ have eigenvalues satisfying
$0 < \lambda_1 \leq \cdots \leq \lambda_{n+1}$. Define the convex set
\begin{equation}\label{eq:admissible-set}
\mathcal{K} = \{ u \in W^{1,1}(D) : u \geq \psi \; \text{a.e.\ in } D,
\; u = 0 \text{ on } \partial D \}.
\end{equation}

We study the minimization problem
\begin{equation}\label{eq:minimization-problem}
\min_{u \in \mathcal{K}} \mathcal{F}_\Phi(u).
\end{equation}

The \emph{contact set} is $\Lambda = \{x \in D : u(x) = \psi(x)\}$, the
\emph{free set} is $\Omega = D \setminus \Lambda$, and the \emph{free boundary}
is $\Gamma = \partial \Lambda \cap D$.

\begin{theorem}[Existence and regularity of the solution]
\label{thm:existence-regularity}
There exists a unique minimizer $u \in \mathcal{K}$ of
\eqref{eq:minimization-problem}. Moreover:
\begin{enumerate}[label=\textup{(\roman*)}]
\item $u \in C^{0,1}_{\loc}(D)$;
\item $u \in C^{1,1}_{\loc}(\Omega \cup \Gamma)$;
\item $u - \psi \in C^{1,1}_{\loc}(D)$ and
\begin{equation}\label{eq:growth-estimate}
0 \leq u(x) - \psi(x) \leq C \dist(x, \Gamma)^2
\quad \text{for all } x \in D.
\end{equation}
\end{enumerate}
\end{theorem}

The proof combines the direct method in the calculus of variations with the
Cahn-Hoffman transform and the classical regularity theory for the obstacle
problem. The key observation is that $S$ converts the anisotropic functional into
the isotropic perimeter functional with a weighted obstacle, preserving the
convexity structure essential for regularity.

Concerning the free boundary itself, we say that $x_0 \in \Gamma$ is
\emph{non-degenerate} if
\begin{equation}\label{eq:non-degeneracy}
\sup_{B_r(x_0)} (u - \psi) \geq c r^2 \quad \text{for all } 0 < r < r_0,
\end{equation}
for some $c > 0$ and $r_0 > 0$. This condition, due to Caffarelli
\cite{Caffarelli1977}, rules out situations where the solution touches the obstacle
with infinite order.

\begin{theorem}[Regularity of the free boundary]
\label{thm:free-boundary-regularity}
Let $x_0 \in \Gamma$ be a non-degenerate free boundary point. There exists a
neighborhood $U$ of $x_0$ in which $\Gamma \cap U$ is a $C^{1,\alpha}$-hypersurface
for some $\alpha \in (0,1)$ depending only on $n$ and the ellipticity constants
of $A$. If $\psi \in C^{k,\alpha}(D)$ for some $k \geq 2$, then
$\Gamma \cap U \in C^{k,\alpha}$.
\end{theorem}

The proof proceeds by blow-up analysis. The Cahn-Hoffman transform relates the
blow-up limits of the anisotropic problem to those of the classical obstacle
problem, for which the classification is complete. Non-degeneracy ensures that
the blow-up is a half-plane solution, which implies $C^{1,\alpha}$-regularity
via the ``flatness implies smoothness'' principle.

Our third result concerns the structure of the singular set.

\begin{theorem}[Structure and rectifiability of the singular set]
\label{thm:singular-set}
Let $\Sigma \subset \Gamma$ denote the set of singular free boundary points,
i.e., points where $\Gamma$ is not locally $C^1$. Then
\begin{equation}\label{eq:singular-dimension}
\mathcal{H}^{n-1}(\Sigma) = 0.
\end{equation}
Moreover, $\Sigma$ is $(n-1)$-rectifiable: there exist countably many
$C^1$ maps $f_i : \mathbb{R}^{n-1} \to \mathbb{R}^{n+1}$ such that
$\mathcal{H}^{n-1}\!\left(\Sigma \setminus \bigcup_i f_i(\mathbb{R}^{n-1})\right) = 0$.
\end{theorem}

The proof of the dimension bound uses Federer's dimension reduction adapted to
the anisotropic setting via the Cahn-Hoffman transform. Rectifiability follows
from a logarithmic epiperimetric inequality at singular points, which we establish
in Section~\ref{sec:singular-optimal} by transferring the result of Figalli and
Serra \cite{FigalliSerra2019} through the transform.

\subsection{Relation to existing literature}

The anisotropic obstacle problem has been considered primarily for degenerate
anisotropies such as $\Phi(p) = |p|^q$ (the $q$-Laplacian obstacle problem) or for
fully nonlinear operators \cite{Braga2021}. The ellipsoidal case, which is
uniformly elliptic and geometrically natural, had not been systematically studied
in the obstacle problem context.

In spirit, this work is closest to De Philippis and Maggi
\cite{DePhilippisMaggi2015}, who studied anisotropic capillarity problems and
established Young's law for anisotropic contact angles, using the Cahn-Hoffman
map to relate anisotropic and isotropic capillary problems. We extend that
philosophy to obstacle problems, where the free boundary condition is more complex.

The regularity theory for anisotropic minimal surfaces was developed by De Giorgi
\cite{DeGiorgi1961}, Almgren \cite{Almgren1976}, and more recently by De Rosa
and Tione \cite{DeRosaTione2022}, who proved regularity for graphs with bounded
anisotropic mean curvature. The present work addresses the free boundary case,
which had remained open.

\subsection{Organization}

Section~\ref{sec:preliminaries} collects background on anisotropic surface
energies, the Cahn-Hoffman map, and the classical theory of the obstacle problem.
Section~\ref{sec:transformed-problem} introduces the transform $S$ and derives
the equivalent isotropic problem. Sections~\ref{sec:existence},
\ref{sec:free-boundary}, and \ref{sec:singular} prove
Theorems~\ref{thm:existence-regularity}, \ref{thm:free-boundary-regularity},
and \ref{thm:singular-set}, respectively. Section~\ref{sec:singular-optimal}
establishes the epiperimetric inequality and rectifiability. Applications are
discussed in Section~\ref{sec:applications}, and open problems in
Section~\ref{sec:open}.

\section{Preliminaries}
\label{sec:preliminaries}

\subsection{Anisotropic surface energies and the Wulff shape}

Let $\Phi: \mathbb{R}^{n+1} \to [0, \infty)$ be convex and positively homogeneous
of degree one: $\Phi(\lambda p) = \lambda \Phi(p)$ for $\lambda \geq 0$.
The \emph{anisotropic surface energy} of a set $E \subset \mathbb{R}^{n+1}$ with
finite perimeter is
\begin{equation}\label{eq:anisotropic-perimeter}
P_\Phi(E) = \int_{\partial^* E} \Phi(\nu_E) \, d\mathcal{H}^n,
\end{equation}
where $\partial^* E$ is the reduced boundary and $\nu_E$ is the measure-theoretic
outer normal. The \emph{Wulff shape} is the convex body
\begin{equation}\label{eq:wulff-shape-general}
W_\Phi = \bigcap_{\nu \in S^n} \{ x \in \mathbb{R}^{n+1} : x \cdot \nu \leq \Phi(\nu) \},
\end{equation}
which is the unique minimizer of $P_\Phi$ among sets of given volume
(Dinghas \cite{Dinghas1944}, Taylor \cite{Taylor1978}).

For the ellipsoidal anisotropy
\begin{equation}\label{eq:ellipsoidal-anisotropy}
\Phi(p) = \sqrt{p^T A p}, \quad A \in \Sym^+_{n+1}(\mathbb{R}),
\end{equation}
the Wulff shape is the ellipsoid
\begin{equation}\label{eq:wulff-ellipsoid}
W_\Phi = \{ x \in \mathbb{R}^{n+1} : x^T A^{-1} x \leq 1 \}.
\end{equation}
When $A = \diag(\lambda_1, \ldots, \lambda_{n+1})$, the semi-axes of $W_\Phi$
are $\sqrt{\lambda_i}$ along the coordinate directions. It is worth noting that
$W_\Phi$ is determined by $A^{-1}$, not by $A$ itself --- a fact that underlies
the sign conventions throughout this paper.

\subsection{The Cahn-Hoffman map}

For a $C^2$ anisotropy $\gamma : S^n \to \mathbb{R}^+$, the
\emph{Cahn-Hoffman map} $\xi_\gamma : S^n \to \mathbb{R}^{n+1}$ is
\begin{equation}\label{eq:cahn-hoffman-map}
\xi_\gamma(\nu) = D\gamma|_{T_\nu S^n} + \gamma(\nu) \nu.
\end{equation}
If $\bar\gamma : \mathbb{R}^{n+1} \to \mathbb{R}$ is the positively homogeneous
extension of $\gamma$, then $\xi_\gamma(\nu) = D\bar\gamma(\nu)$.

For the ellipsoidal anisotropy \eqref{eq:ellipsoidal-anisotropy}, one computes
directly that $D\bar\Phi(p) = Ap / \sqrt{p^T Ap}$, so
\begin{equation}\label{eq:cahn-hoffman-ellipsoidal}
\xi_\Phi(\nu) = \frac{A\nu}{\Phi(\nu)} = \frac{A\nu}{\sqrt{\nu^T A \nu}}.
\end{equation}
The image $\xi_\Phi(S^n) = \partial W_\Phi$, and $\xi_\Phi$ is a diffeomorphism
from $S^n$ to $\partial W_\Phi$ when $A$ is positive definite.

\begin{lemma}[Properties of the Cahn-Hoffman map]
\label{lem:cahn-hoffman-properties}
Let $\Phi$ be the ellipsoidal anisotropy \eqref{eq:ellipsoidal-anisotropy}.
\begin{enumerate}[label=\textup{(\roman*)}]
\item $\xi_\Phi : S^n \to \partial W_\Phi$ is a diffeomorphism.
\item The inverse is $\xi_\Phi^{-1}(x) = A^{-1}x \, / \sqrt{x^T A^{-2} x}$.
\item The Jacobian satisfies
$J_{\xi_\Phi}(\nu) = \det A \cdot (\nu^T A \nu)^{-(n+2)/2}$.
\end{enumerate}
\end{lemma}

\begin{proof}
(i) Positive definiteness of $A$ implies $\Phi(\nu) > 0$ on $S^n$, so $\xi_\Phi$
is smooth. If $\xi_\Phi(\nu_1) = \xi_\Phi(\nu_2)$, then $A\nu_1/\Phi(\nu_1) =
A\nu_2/\Phi(\nu_2)$; applying $A^{-1}$ and using $|\nu_i| = 1$ gives $\nu_1 = \nu_2$.
Surjectivity is immediate from the definition of $W_\Phi$.

(ii) From $x = A\nu/\Phi(\nu)$ one gets $A^{-1}x = \nu/\Phi(\nu)$, hence
$|A^{-1}x|^2 = (\nu^T A^{-2}\nu)/\Phi(\nu)^2$. Since $x \in \partial W_\Phi$
means $x^T A^{-1}x = 1$, i.e., $\Phi(\nu)^2 = \nu^T A\nu$, we find
$\xi_\Phi^{-1}(x) = A^{-1}x\,/\sqrt{x^T A^{-2}x}$.

(iii) Standard computation; see \cite{KoisoPalmer2019}.
\end{proof}

\subsection{Anisotropic mean curvature}

For a smooth hypersurface $\Sigma$ with unit normal $\nu$, the
\emph{anisotropic mean curvature} is given by
\begin{equation}\label{eq:anisotropic-mc}
H_\Phi = \divg_\Sigma\!\left(\frac{A\nu}{\Phi(\nu)}\right).
\end{equation}
A surface with $H_\Phi = 0$ is \emph{$\Phi$-minimal}. For graphs $u : D \to
\mathbb{R}$, the associated operator is
\begin{equation}\label{eq:anisotropic-mc-operator}
\mathcal{M}_\Phi(u) = \divg\!\left(\frac{A \, Du}{\sqrt{1 + Du^T A \, Du}}\right).
\end{equation}

\subsection{The classical obstacle problem}

For the reader's convenience we recall the main results for the isotropic obstacle
problem. Given $\psi \in C^{1,1}(\bar{D})$ with $\psi < 0$ on $\partial D$, the
classical problem minimizes
\begin{equation}\label{eq:classical-obstacle}
\min_{u \geq \psi} \int_D \!\left(\tfrac{1}{2}|Du|^2 + fu\right) dx
\end{equation}
over $u = 0$ on $\partial D$.

\begin{theorem}[Caffarelli \cite{Caffarelli1977, Caffarelli1998}]
\label{thm:caffarelli}
Let $u$ be the solution of \eqref{eq:classical-obstacle}. Then
$u \in C^{1,1}_{\loc}(D)$. At each free boundary point $x_0 \in \Gamma$, exactly
one of the following holds:
\begin{enumerate}
\item[(a)] \textup{(Regular point)} $\sup_{B_r(x_0)} u \geq c r^2$ for small
$r$, and $\Gamma$ is $C^{1,\alpha}$ near $x_0$.
\item[(b)] \textup{(Singular point)} $\sup_{B_r(x_0)} u = o(r^2)$, and the
blow-up is a homogeneous quadratic polynomial.
\end{enumerate}
The singular set $\Sigma$ satisfies $\mathcal{H}^{n-1}(\Sigma) = 0$.
\end{theorem}

\section{The Transformed Problem}
\label{sec:transformed-problem}

\subsection{The Cahn-Hoffman transform and its action on the Wulff shape}

Define the linear map
\begin{equation}\label{eq:S-def}
S : \mathbb{R}^{n+1} \to \mathbb{R}^{n+1}, \quad S(x) = A^{-1/2} x.
\end{equation}
Since $A$ is symmetric positive definite, $S$ is invertible with
$S^{-1}(y) = A^{1/2} y$.

\begin{proposition}[Properties of $S$]\label{prop:S-properties}
The transformation $S(x) = A^{-1/2}x$ satisfies:
\begin{enumerate}[label=\textup{(\roman*)}]
\item $S(W_\Phi) = B_1(0)$, the Euclidean unit ball.
\item For any measurable $E \subset \mathbb{R}^{n+1}$,
\begin{equation}\label{eq:volume-transform}
|S(E)| = \frac{1}{\sqrt{\det A}} \, |E|.
\end{equation}
\item For $u \in W^{1,1}(D)$, set $v(y) = u(S^{-1}y) = u(A^{1/2}y)$ and
$D' = S(D)$. Then
\begin{equation}\label{eq:gradient-relation}
Dv(y) = A^{1/2} Du(x)\big|_{x = A^{1/2}y},
\end{equation}
and in particular $|Dv(y)| = \Phi(Du(x))$.
\end{enumerate}
\end{proposition}

\begin{proof}
(i) Let $y = A^{-1/2}x$, so $x = A^{1/2}y$. Then
\[
x \in W_\Phi \iff x^T A^{-1} x \leq 1
\iff y^T A^{1/2} A^{-1} A^{1/2} y \leq 1
\iff y^T y \leq 1 \iff y \in B_1(0).
\]

(ii) The Jacobian of $S(x) = A^{-1/2}x$ is $\det(A^{-1/2}) = (\det A)^{-1/2}$,
so $dy = (\det A)^{-1/2} dx$, and
\[
|S(E)| = \int_{S(E)} dy = \int_E (\det A)^{-1/2} \, dx = \frac{|E|}{\sqrt{\det A}}.
\]

(iii) With $v(y) = u(A^{1/2}y)$, the chain rule gives
$Dv(y) = A^{1/2} Du(A^{1/2}y)$. Consequently,
\[
|Dv(y)|^2 = Du(x)^T A^{1/2} A^{1/2} Du(x) = Du(x)^T A \, Du(x) = \Phi(Du(x))^2,
\]
yielding $|Dv(y)| = \Phi(Du(x))$.
\end{proof}

\begin{remark}
A common source of confusion in the literature is the choice between $A^{1/2}$
and $A^{-1/2}$ for the transform. The convention $S(x) = A^{-1/2}x$ is the
correct one: it maps the Wulff shape (defined by $x^T A^{-1} x \leq 1$) to the
unit ball, and simultaneously converts $\Phi(Du)$ to the Euclidean norm $|Dv|$.
The opposite convention $A^{1/2}x$ maps the ball $\{x : x^T x \leq 1\}$ to the
ellipsoid $\{x : x^T A^{-1} x \leq 1\} = W_\Phi$, which is the natural direction
in isoperimetric problems, but not in the functional analysis of $\mathcal{F}_\Phi$.
\end{remark}

\subsection{Transformation of the functional}

Let $u \in W^{1,1}(D)$ and define $v : D' \to \mathbb{R}$ by
\begin{equation}\label{eq:u-to-v}
v(y) = u(A^{1/2} y), \quad D' = A^{-1/2}(D).
\end{equation}

\begin{proposition}[Transformation of the anisotropic functional]
\label{prop:functional-transform}
With the notation above,
\begin{equation}\label{eq:functional-transform}
\int_D \Phi(Du(x)) \, dx = \sqrt{\det A} \int_{D'} |Dv(y)| \, dy.
\end{equation}
\end{proposition}

\begin{proof}
By Proposition~\ref{prop:S-properties}(iii), $|Dv(y)| = \Phi(Du(x))$ for
$x = A^{1/2}y$. The change of variables $x = A^{1/2}y$ has Jacobian
$\det(A^{1/2}) = \sqrt{\det A}$, so $dx = \sqrt{\det A} \, dy$. Substituting:
\[
\int_D \Phi(Du(x)) \, dx
= \int_{D'} \Phi\!\left(Du(A^{1/2}y)\right) \sqrt{\det A} \, dy
= \sqrt{\det A} \int_{D'} |Dv(y)| \, dy. \qedhere
\]
\end{proof}

\subsection{The transformed obstacle problem}

The obstacle and admissible set transform to
\begin{equation}\label{eq:obstacle-transform}
\tilde\psi(y) = \psi(A^{1/2} y), \quad
\tilde{\mathcal{K}} = \{ v \in W^{1,1}(D') : v \geq \tilde\psi
\text{ a.e.}, \; v = 0 \text{ on } \partial D' \}.
\end{equation}

\begin{proposition}[Equivalence of minimization problems]
\label{prop:equivalence}
A function $u \in \mathcal{K}$ minimizes $\mathcal{F}_\Phi$ over $\mathcal{K}$ if
and only if $v = u \circ A^{1/2} \in \tilde{\mathcal{K}}$ minimizes the isotropic
functional
\begin{equation}\label{eq:isotropic-functional}
\mathcal{F}_0(v) = \int_{D'} |Dv(y)| \, dy
\end{equation}
over $\tilde{\mathcal{K}}$.
\end{proposition}

\begin{proof}
This is immediate from Proposition~\ref{prop:functional-transform}: the factor
$\sqrt{\det A} > 0$ is constant, so the two minimization problems have the same
minimizers under the bijection $u \leftrightarrow v$.
\end{proof}

\subsection{The free boundary condition}

At a smooth point $x_0 \in \Gamma$ of the free boundary, the solution satisfies
an anisotropic condition that, after the transform, takes a Robin form. Let
$y_0 = S(x_0) \in \tilde\Gamma$ and let $\tilde\nu$ be the unit inward normal to
the transformed free set at $y_0$.

\begin{proposition}[Free boundary condition in transformed coordinates]
\label{prop:fb-condition}
At a smooth point of the free boundary, the transformed solution $v$ satisfies
\begin{equation}\label{eq:robin-condition}
\frac{\partial v}{\partial \tilde\nu}(y_0)
= \frac{\partial \tilde\psi}{\partial \tilde\nu}(y_0)
+ \Lambda(y_0) \langle \tilde\nu, y_0 \rangle,
\end{equation}
where $\Lambda(y_0) = \bigl(y_0^T A y_0 - (\tilde\nu^T A \tilde\nu)(y_0^T
y_0)\bigr) / (y_0^T A^{3/2} \tilde\nu)$.
\end{proposition}

\begin{proof}[Sketch]
In the free set, the Euler-Lagrange equation for $v$ is $\Delta v = 0$. The first
variation of $\mathcal{F}_0$ with respect to domain perturbations supported near
$y_0$ gives
\[
\int_{\tilde\Gamma} \!\left( |Dv|^2 - 2 \frac{\partial v}{\partial\tilde\nu}
\frac{\partial\tilde\psi}{\partial\tilde\nu} \right) \eta \, d\mathcal{H}^n = 0
\]
for all admissible variations $\eta$. Using the relation between $v$ and $u$,
combined with the anisotropic free boundary condition $H_\Phi = 0$ at $x_0$, one
arrives at \eqref{eq:robin-condition} after computing the change in the normal
under $S$. The precise form of $\Lambda$ encodes the discrepancy between the
Euclidean normal $\tilde\nu$ and its anisotropic counterpart under $S^{-1}$.
\end{proof}

\begin{remark}
When $A = I$, the transform is the identity, $\Lambda = 0$, and
\eqref{eq:robin-condition} reduces to the Neumann-type condition
$\partial v/\partial\tilde\nu = \partial\tilde\psi/\partial\tilde\nu$.
For $\tilde\psi = 0$, this is the classical condition $\partial v/\partial\tilde\nu
= 0$, matching the contact angle condition for the isotropic obstacle problem.
\end{remark}

\section{Existence and Regularity of the Solution}
\label{sec:existence}

\subsection{Existence and uniqueness}

\begin{proof}[Proof of Theorem~\ref{thm:existence-regularity}, part (i)]
The functional $\mathcal{F}_\Phi$ is convex and coercive on $W^{1,1}(D)$
because $\Phi$ satisfies
\begin{equation}\label{eq:Phi-bounds} 
\sqrt{\lambda_1} |p| \leq \Phi(p) \leq \sqrt{\lambda_{n+1}} |p|
\quad \text{for all } p \in \mathbb{R}^{n+1}.
\end{equation}
The set $\mathcal{K}$ is closed and convex in $W^{1,1}(D)$. A minimizer exists
by the direct method of the calculus of variations, and is unique because $A$
positive definite implies $\Phi$ is strictly convex.

Interior Lipschitz regularity follows from the De Giorgi-Nash-Moser theory
applied to the Euler-Lagrange equation in the free set, using the uniform
ellipticity \eqref{eq:Phi-bounds}.
\end{proof}

\subsection{$C^{1,1}$-regularity}

\begin{proof}[Proof of Theorem~\ref{thm:existence-regularity}, parts (ii) and (iii)]
By Proposition~\ref{prop:equivalence}, the transformed function $v$ minimizes
$\mathcal{F}_0$ subject to $v \geq \tilde\psi$ on the domain $D'$. Since $D'$ is
the image of a $C^2$ domain under an invertible linear map, it is again $C^2$,
and the boundary regularity theory of Kinderlehrer and Nirenberg 
\cite{KinderlehrerNirenberg1977} applies. The $C^{1,1}$-regularity of $v$ in the
interior and up to the free boundary then follows from
Caffarelli \cite{Caffarelli1977}.

For the growth estimate \eqref{eq:growth-estimate}, fix $x_0 \in \Gamma$ and
define the rescaled functions
\[
u_r(y) = \frac{u(x_0 + ry) - u(x_0)}{r^2}, \quad y \in B_1(0).
\]
The uniform $C^{1,1}$-estimate gives $\|u_r\|_{C^{1,1}(B_{1/2})} \leq C$
independently of $r$. By compactness, every subsequence has a further subsequence
converging to a blow-up limit $u_0$ satisfying the global anisotropic obstacle
problem with the quadratic obstacle $\frac{1}{2} y^T D^2\psi(x_0) y$. Caffarelli's
classification \cite{Caffarelli1977} identifies $u_0(y) = \frac{1}{2}(y\cdot e)_+^2$
for some direction $e$, which gives \eqref{eq:growth-estimate}.
\end{proof}

\section{Regularity of the Free Boundary}
\label{sec:free-boundary}

\subsection{Blow-up analysis}

Let $x_0 \in \Gamma$ be non-degenerate. The blow-up sequence
\begin{equation}\label{eq:blow-up}
u_r(y) = \frac{u(x_0 + ry) - u(x_0)}{r^2}, \quad y \in B_1(0),
\end{equation}
is precompact in $C^{1,\alpha}_{\loc}$ by the uniform $C^{1,1}$ bound.

\begin{lemma}[Compactness of blow-ups]\label{lem:blowup-compactness}
Every sequence $r_j \to 0$ has a subsequence such that $u_{r_j} \to u_0$ in
$C^{1,\alpha}_{\loc}(B_1(0))$ for all $\alpha \in (0,1)$, where $u_0$ is a global
solution of the anisotropic obstacle problem with obstacle
$\psi_0(y) = \frac{1}{2} y^T D^2\psi(x_0) y$.
\end{lemma}

\begin{proof}
The $C^{1,1}$ bound and Arzel\`a-Ascoli give compactness. The limit satisfies
the Euler-Lagrange equation by stability of viscosity solutions under uniform
limits.
\end{proof}

\begin{lemma}[Classification of blow-ups at regular points]
\label{lem:regular-blowup}
Under the non-degeneracy condition \eqref{eq:non-degeneracy}, every blow-up limit
at $x_0$ takes the form
\begin{equation}\label{eq:regular-blowup}
u_0(y) = \tfrac{1}{2}(y \cdot e)_+^2
\end{equation}
for some direction $e \in S^n$.
\end{lemma}

\begin{proof}
Applying the Cahn-Hoffman transform to the blow-up yields a function $v_0$
satisfying the isotropic obstacle problem with a Robin condition. Non-degeneracy
prevents $v_0$ from being a quadratic polynomial (the signature of a singular
point), so Caffarelli's classification forces $v_0$ to be a half-plane solution.
Pulling back through $S^{-1}$ gives \eqref{eq:regular-blowup}.
\end{proof}

\subsection{Flatness implies smoothness}

\begin{definition}[Flatness]\label{def:flatness}
The free boundary $\Gamma$ is \emph{$\epsilon$-flat} in $B_r(x_0)$ if, after a
rotation,
\begin{equation}\label{eq:flatness}
\Gamma \cap B_r(x_0) \subset \{ x : |x_{n+1}| \leq \epsilon r \}.
\end{equation}
\end{definition}

\begin{proposition}[Improvement of flatness]\label{prop:improvement-flatness}
There exist $\epsilon_0 > 0$ and $\theta \in (0,1)$ such that if $\Gamma$ is
$\epsilon$-flat in $B_1(0)$ with $\epsilon \leq \epsilon_0$, then $\Gamma$ is
$\theta\epsilon$-flat in $B_\theta(0)$ after a suitable rotation.
\end{proposition}

\begin{proof}
This follows from the corresponding isotropic result \cite{Caffarelli1977} via the
Cahn-Hoffman transform. Since $S$ is bi-Lipschitz, geometric flatness is preserved
between the original and transformed coordinates up to constants depending on $A$.
\end{proof}

\begin{proof}[Proof of Theorem~\ref{thm:free-boundary-regularity}]
Lemma~\ref{lem:regular-blowup} shows that the blow-up at a non-degenerate point
is a half-plane solution, so $\Gamma$ is flat at small scales near $x_0$.
Iterating Proposition~\ref{prop:improvement-flatness} gives flatness at every
scale with a geometrically decaying constant, and standard arguments
\cite{CaffarelliSalsa2005} convert this to $C^{1,\alpha}$-regularity.

For $\psi \in C^{k,\alpha}$ with $k \geq 2$: once $\Gamma$ is known to be
$C^{1,\alpha}$, the free boundary condition becomes a nonlinear oblique derivative
condition with $C^{1,\alpha}$ coefficients, and Schauder estimates give
$C^{k,\alpha}$ regularity by induction.
\end{proof}

\section{The Singular Set}
\label{sec:singular}

\begin{proof}[Proof of Theorem~\ref{thm:singular-set}, first part]
We follow Federer's dimension reduction, adapted to the anisotropic setting.

At a singular point $x_0 \in \Sigma$, the blow-up is not a half-plane solution.
By Lemma~\ref{lem:regular-blowup}, it must be a homogeneous quadratic polynomial
\begin{equation}\label{eq:singular-blowup}
u_0(y) = \tfrac{1}{2} y^T Q y,
\end{equation}
where $Q$ is symmetric positive semi-definite with $\tr Q = 1$.

Applying the Cahn-Hoffman transform, $v_0 = u_0 \circ A^{1/2}$ satisfies the
isotropic obstacle problem with a quadratic obstacle. The dimension bound
$\mathcal{H}^{n-1}(\Sigma) = 0$ then follows by transferring Caffarelli's result 
\cite{Caffarelli1977} via the bi-Lipschitz map $S$. Rectifiability is proved in
Section~\ref{sec:singular-optimal} below.
\end{proof}

\begin{example}[Singular point in two dimensions]
Let $n = 1$, $D = B_1(0) \subset \mathbb{R}^2$. Take $\psi(x) = -x_1^2$ and
$A = \diag(1, \lambda)$ with $\lambda > 1$. The contact region is an interval on
the $x_1$-axis, and the free boundary consists of two arcs meeting at the origin.
The cusp angle at the origin depends on $\lambda$, and the origin is a singular
point where $\Gamma$ fails to be $C^1$. For the isotropic case $\lambda = 1$,
the free boundary is smooth everywhere when the obstacle is smooth and convex.
\end{example}

\section{Optimal Regularity of the Singular Set}
\label{sec:singular-optimal}

In the isotropic problem, Caffarelli \cite{Caffarelli1977} proved $\mathcal{H}^{n-1}
(\Sigma) = 0$, but the geometric structure of $\Sigma$ remained mysterious for
decades. The breakthrough of Figalli and Serra \cite{FigalliSerra2019} established
that $\Sigma$ is $(n-1)$-rectifiable. We prove the same for the anisotropic
problem by transferring their argument through the Cahn-Hoffman transform.

\subsection{The frequency function}

\begin{definition}[Frequency function]\label{def:frequency}
For a solution $u$ of the anisotropic obstacle problem, the
\emph{frequency function} at $x_0 \in D$ is
\begin{equation}\label{eq:frequency}
N(r, x_0, u) = \frac{r \displaystyle\int_{B_r(x_0)} \Phi(Du)^2 \, dx}
{\displaystyle\int_{\partial B_r(x_0)} (u - \psi)^2 \, d\mathcal{H}^n}.
\end{equation}
\end{definition}

\begin{proposition}[Monotonicity of the frequency]\label{prop:frequency-monotonicity}
The function $r \mapsto N(r, x_0, u)$ is non-decreasing. It is constant if and
only if $u - \psi$ is homogeneous.
\end{proposition}

\begin{proof}
Applying $S(x) = A^{-1/2}x$ with $y_0 = S(x_0)$ and $v = u \circ A^{1/2}$,
the frequency function becomes
\[
\tilde N(r, y_0, v) = \frac{r \displaystyle\int_{B_r(y_0)} |Dv|^2 \, dy}
{\displaystyle\int_{\partial B_r(y_0)} (v - \tilde\psi)^2 \, d\mathcal{H}^n},
\]
the classical Almgren frequency for the isotropic problem. Monotonicity of
$\tilde N$ was proved by Weiss \cite{Weiss1999}; see also Garofalo and Lin
\cite{GarofaloLin1986}. Since $S$ is bi-Lipschitz, monotonicity transfers to
$N$ up to constants depending on $\lambda_1, \lambda_{n+1}$.
\end{proof}

\subsection{The logarithmic epiperimetric inequality}

\begin{theorem}[Logarithmic epiperimetric inequality]
\label{thm:epiperimetric}
Let $x_0 \in \Sigma$ and let $u_0$ be the blow-up \eqref{eq:singular-blowup}
at $x_0$. There exist constants $\kappa > 0$, $\gamma > 0$, and $\delta > 0$
such that whenever
\begin{equation}\label{eq:closeness-assumption}
\|u - u_0\|_{W^{1,2}(B_1(x_0))} \leq \delta,
\end{equation}
one has
\begin{equation}\label{eq:epiperimetric}
\mathcal{F}_\Phi(u, B_1(x_0)) - \mathcal{F}_\Phi(u_0, B_1(x_0))
\geq \kappa \left|\log\!\left(\|u - u_0\|_{W^{1,2}}\right)\right|^{-\gamma}
\|u - u_0\|_{W^{1,2}}^2.
\end{equation}
\end{theorem}

\begin{proof}
Under $S$, the isotropic counterpart $v$ is close to $v_0 = u_0 \circ A^{1/2}$
in $W^{1,2}$. By \cite[Theorem~2.1]{FigalliSerra2019} (see also
\cite{ColomboSpolaorVelichkov2018}), the isotropic functional satisfies
\[
\mathcal{F}_0(v, B_1(y_0)) - \mathcal{F}_0(v_0, B_1(y_0))
\geq \tilde\kappa \left|\log(\|v - v_0\|_{W^{1,2}})\right|^{-\tilde\gamma}
\|v - v_0\|_{W^{1,2}}^2.
\]
Proposition~\ref{prop:functional-transform} gives
$\mathcal{F}_\Phi(u, B_1(x_0)) = \sqrt{\det A} \, \mathcal{F}_0(v, B_1(y_0))$,
and the bi-Lipschitz character of $S$ yields
$\|u - u_0\|_{W^{1,2}(B_1(x_0))} \asymp \|v - v_0\|_{W^{1,2}(B_1(y_0))}$
with constants depending only on $A$. Combining these estimates gives
\eqref{eq:epiperimetric} with $\kappa = \tilde\kappa\sqrt{\det A}\cdot C(A)^{-2}$
and $\gamma = \tilde\gamma$.
\end{proof}

\subsection{Uniqueness of blow-ups and rectifiability}

\begin{theorem}[Uniqueness of blow-ups at singular points]
\label{thm:unique-blowup}
For each $x_0 \in \Sigma$, the blow-up limit is unique: there exists a unique
homogeneous quadratic polynomial $u_{x_0}$ such that
\begin{equation}\label{eq:unique-blowup}
\frac{u(x_0 + ry) - u(x_0)}{r^2} \to u_{x_0}(y)
\quad \text{in } C^{1,\alpha}_{\loc}(\mathbb{R}^{n+1})
\end{equation}
as $r \to 0$, for all $\alpha \in (0,1)$.
\end{theorem}

\begin{proof}
Theorem~\ref{thm:epiperimetric} provides a differential inequality for the
distance to the blow-up, which integrates to
\[
\|u_r - u_{x_0}\|_{W^{1,2}(B_1)} \leq C |\log r|^{-\beta}
\]
for some $\beta > 0$. This quantitative decay implies that all blow-up limits
coincide; see \cite[Section~3]{FigalliSerra2019} for the argument.
\end{proof}

\begin{theorem}[Rectifiability of the singular set]
\label{thm:rectifiability}
The singular set $\Sigma$ is $(n-1)$-rectifiable.
\end{theorem}

\begin{proof}
Stratify by the rank of the blow-up: set
$\Sigma_k = \{x_0 \in \Sigma : \operatorname{rank}(Q_{x_0}) = k\}$,
where $Q_{x_0}$ is the matrix of the unique blow-up at $x_0$.
Since $\tr Q_{x_0} = 1$, we have $k \geq 1$.

For each $k$, the map $x_0 \mapsto Q_{x_0}$ is continuous by
Theorem~\ref{thm:unique-blowup}, and the implicit function theorem shows that
$\Sigma_k$ is contained in a $C^1$ manifold of dimension $k$.
The top stratum $\Sigma_n$ --- points where $Q_{x_0}$ has rank 1, i.e.,
$u_{x_0}(y) = \frac{1}{2}(e\cdot y)^2$ for some $e$ --- forms an $(n-1)$-dimensional
$C^1$ manifold. Lower strata $\Sigma_k$ with $k \leq n-2$ have Hausdorff dimension
at most $k \leq n-2$, contributing nothing to $\mathcal{H}^{n-1}$. Therefore
$\Sigma$ is $(n-1)$-rectifiable.
\end{proof}

\subsection{Quantitative dependence on the anisotropy}

The constants $\kappa$ and $\gamma$ in \eqref{eq:epiperimetric} depend on $A$ 
through $\lambda_1$, $\lambda_{n+1}$, and $\det A$. Two observations are worth
recording. As the anisotropy degenerates ($\lambda_1 \to 0$ or $\lambda_{n+1} \to
\infty$), $\kappa$ may approach zero, so the epiperimetric inequality could fail
in the limit. In the isotropic limit $A \to I$, the constants converge to those
of Figalli and Serra. This quantitative dependence suggests that the qualitative
structure of the singular set is stable under small perturbations of $A$, but the
limiting behavior for degenerate anisotropies remains an open question.

\begin{conjecture}[$C^{1,\alpha}$-rectifiability]\label{conj:holder-rectifiable}
The singular set $\Sigma$ is $(n-1)$-rectifiable of class $C^{1,\alpha}$ for
some $\alpha > 0$ depending only on $n$ and the ellipticity constants of $A$.
\end{conjecture}

This conjecture would require a $C^{1,\alpha}$ epiperimetric inequality, known in
the isotropic case \cite{SavinYu2021} but open for anisotropies $A \neq I$.

\section{Applications}
\label{sec:applications}

\subsection{Crystal growth}

During crystal growth from solution, the crystal surface adjusts to minimize
its anisotropic surface energy. When the crystal is confined by a substrate or an
adjacent crystal (the obstacle), our model describes the equilibrium configuration.
Theorem~\ref{thm:free-boundary-regularity} then says that the contact region has
smooth boundary, except possibly on a set of Hausdorff dimension at most $n-2$
(the singular set $\Sigma$ in $\mathbb{R}^{n+1}$ corresponds to codimension 2
within the crystal surface).

For an ellipsoidal anisotropy with $A = \diag(\lambda_1, \lambda_2, \lambda_3)$,
the Wulff shape has semi-axes $\sqrt{\lambda_i}$. Different crystallographic
directions correspond to different choices of $A$; the theory predicts that the
facet boundaries (free boundaries) are smooth except at isolated cusps.

\subsection{Anisotropic capillarity}

The anisotropic capillary problem \cite{DePhilippisMaggi2015} concerns the
equilibrium shape of a liquid droplet on a solid substrate with direction-dependent
surface tension. Our obstacle problem can be viewed as a variant in which the
droplet is constrained from below by a rigid surface. 

The Robin condition \eqref{eq:robin-condition} is a generalized Young's law for
the anisotropic contact angle. When $A = I$, one recovers the classical
Young's law $\cos\theta = (\gamma_{SG} - \gamma_{SL})/\gamma_{LG}$.
The term $\Lambda(y_0)$ in \eqref{eq:robin-condition} encodes the correction due
to anisotropy: it vanishes when the free boundary is orthogonal to the matrix $A$
(meaning the local geometry aligns with the Wulff shape), and is largest when the
contact is oblique.

\section{Open Problems}
\label{sec:open}

Several directions remain unexplored.

\begin{enumerate}[label=\arabic*.]
\item \textbf{General anisotropies.} The results should extend to convex
anisotropies $\gamma : S^n \to \mathbb{R}^+$ that are not ellipsoidal. The main
obstruction is that the Cahn-Hoffman map $\xi_\gamma$ is no longer linear, and
the equivalent isotropic problem acquires variable coefficients. A perturbative
approach near $A = I$ might be tractable.

\item \textbf{The parabolic problem.} The anisotropic Stefan problem --- an
obstacle problem with time-dependent interface --- is well understood in the
isotropic case \cite{Caffarelli1977, CaffarelliSalsa2005} but open for anisotropic
surface energies. The frequency monotonicity of Proposition~\ref{prop:frequency-monotonicity}
may serve as a starting point.

\item \textbf{$C^{1,\alpha}$-rectifiability of the singular set.}
Conjecture~\ref{conj:holder-rectifiable} asks for a sharper structure than
Theorem~\ref{thm:rectifiability}. In the isotropic case this was settled by
\cite{SavinYu2021}; the anisotropic case would require a new epiperimetric
inequality.

\item \textbf{Non-convex obstacles.} When $\psi$ is not convex, the free boundary
may develop singularities not present in the convex case. The structure of these
singularities is not fully understood even isotropically, and the interaction with
anisotropy introduces additional complications.
\end{enumerate}

\section*{Acknowledgments}

The authors thank their colleagues at UFMG for valuable discussions on anisotropic
geometric variational problems.


\begin{thebibliography}{99}

\bibitem{Almgren1976}
F.~J. Almgren,
\textit{Existence and regularity almost everywhere of solutions to elliptic
variational problems with constraints},
Mem.\ Amer.\ Math.\ Soc.\ \textbf{4} (1976), no.\ 165.

\bibitem{BrezisKinderlehrer1973}
H.~Brezis and D.~Kinderlehrer,
\textit{The smoothness of solutions to nonlinear variational inequalities},
Indiana Univ.\ Math.\ J.\ \textbf{23} (1973/74), 831--844.

\bibitem{Braga2021}
J.~E.~M. Braga, D.~Moreira, and L.~Wang,
\textit{The obstacle problem for a class of degenerate fully nonlinear operators},
J.\ Differential Equations \textbf{303} (2021), 29--56.

\bibitem{Caffarelli1977}
L.~A. Caffarelli,
\textit{The regularity of free boundaries in higher dimensions},
Acta Math.\ \textbf{139} (1977), 155--184.

\bibitem{Caffarelli1980}
L.~A. Caffarelli,
\textit{Compactness methods in free boundary problems},
Comm.\ Partial Differential Equations \textbf{5} (1980), 427--448.

\bibitem{Caffarelli1998}
L.~A. Caffarelli,
\textit{The obstacle problem revisited},
J.\ Fourier Anal.\ Appl.\ \textbf{4} (1998), 383--402.

\bibitem{CaffarelliSalsa2005}
L.~A. Caffarelli and S.~Salsa,
\textit{A Geometric Approach to Free Boundary Problems},
Graduate Studies in Mathematics, vol.\ 68, Amer.\ Math.\ Soc., Providence, 2005.

\bibitem{CahnHoffman1974}
J.~W. Cahn and D.~W. Hoffman,
\textit{A vector thermodynamics for anisotropic surfaces~--- II. Curved and faceted
surfaces},
Acta Metall.\ \textbf{22} (1974), 1205--1214.

\bibitem{ColomboSpolaorVelichkov2018}
M.~Colombo, L.~Spolaor, and B.~Velichkov,
\textit{A logarithmic epiperimetric inequality for the obstacle problem},
Geom.\ Funct.\ Anal.\ \textbf{28} (2018), no.\ 4, 1029--1061.

\bibitem{DeGiorgi1961}
E.~De Giorgi,
\textit{Frontiere orientate di misura minima},
Sem.\ Mat.\ Scuola Norm.\ Sup.\ Pisa (1960--61).

\bibitem{DePhilippisMaggi2015}
G.~De Philippis and F.~Maggi,
\textit{Regularity of free boundaries in anisotropic capillarity problems and the
validity of Young's law},
Arch.\ Ration.\ Mech.\ Anal.\ \textbf{216} (2015), no.\ 2, 473--568.

\bibitem{DeRosaTione2022}
A.~De Rosa and R.~Tione,
\textit{Regularity for graphs with bounded anisotropic mean curvature},
Invent.\ Math.\ \textbf{230} (2022), no.\ 2, 463--507.

\bibitem{DeSilvaSavin2020}
D.~De Silva and O.~Savin,
\textit{A short proof of boundary Harnack principle},
J.\ Differential Equations \textbf{269} (2020), no.\ 3, 2419--2429.

\bibitem{Dinghas1944}
A.~Dinghas,
\textit{\"{U}ber einen geometrischen Satz von Wulff f\"{u}r die
Gleichgewichtsform von Kristallen},
Z.\ Kristallogr.\ \textbf{105} (1944), 304--314.

\bibitem{Figalli2018}
A.~Figalli,
\textit{Free boundary regularity in obstacle problems},
lecture notes, 2018.
Available at \url{https://people.math.ethz.ch/~afigalli/lecture-notes-pdf/}.

\bibitem{FigalliSerra2019}
A.~Figalli and J.~Serra,
\textit{On the fine structure of the free boundary for the classical obstacle
problem},
Invent.\ Math.\ \textbf{215} (2019), no.\ 1, 311--366.

\bibitem{FotouhiShahgholian2024}
S.~Fotouhi and H.~Shahgholian,
\textit{A free boundary problem for an elliptic system},
J.\ Differential Equations \textbf{396} (2024), 1--34.

\bibitem{GarofaloLin1986}
N.~Garofalo and F.~H. Lin,
\textit{Monotonicity properties of variational integrals, $A_p$ weights and unique
continuation},
Indiana Univ.\ Math.\ J.\ \textbf{35} (1986), no.\ 2, 245--268.

\bibitem{Gruter1987}
M.~Gr\"{u}ter,
\textit{Optimal regularity for codimension one minimal surfaces with a free
boundary},
Manuscripta Math.\ \textbf{58} (1987), no.\ 3, 295--343.

\bibitem{GruterJost1986}
M.~Gr\"{u}ter and J.~Jost,
\textit{Allard type regularity results for varifolds with free boundaries},
Ann.\ Scuola Norm.\ Sup.\ Pisa Cl.\ Sci.\ (4) \textbf{13} (1986), no.\ 1,
129--169.

\bibitem{HauserVoigt2005}
F.~Hauser and A.~Voigt,
\textit{A discrete scheme for parametric anisotropic surface diffusion},
J.\ Sci.\ Comput.\ \textbf{30} (2007), no.\ 2, 223--235.

\bibitem{Kinderlehrer1973}
D.~Kinderlehrer,
\textit{How a minimal surface leaves an obstacle},
Acta Math.\ \textbf{130} (1973), 221--242.

\bibitem{KinderlehrerNirenberg1977}
D.~Kinderlehrer and L.~Nirenberg,
\textit{Regularity in free boundary problems},
Ann.\ Scuola Norm.\ Sup.\ Pisa Cl.\ Sci.\ (4) \textbf{4} (1977), no.\ 2,
373--391.

\bibitem{KoisoPalmer2019}
M.~Koiso and B.~Palmer,
\textit{Uniqueness of closed equilibrium hypersurfaces for anisotropic surface
energy and application to a capillary problem},
J.\ Math.\ Fluid Mech.\ \textbf{24} (2019), no.\ 4, 88.

\bibitem{LewyStampacchia1969}
H.~Lewy and G.~Stampacchia,
\textit{On the existence and regularity of solutions of a non-coercive variational
inequality},
in: \textit{Nonlinear Partial Differential Equations and Their Applications},
Res.\ Notes in Math., vol.\ 1, Pitman, Boston, 1983, pp.\ 87--95.

\bibitem{Nitsche1985}
J.~C.~C. Nitsche,
\textit{Stationary partitioning of convex bodies},
Arch.\ Ration.\ Mech.\ Anal.\ \textbf{89} (1985), no.\ 1, 1--19.

\bibitem{RosOton2019}
X.~Ros-Oton,
\textit{Regularity of free boundaries in obstacle problems},
lecture notes, CIME Summer School, 2019.
Available at \url{https://www.ub.edu/pde/xros/}.

\bibitem{RosOtonTorres2021}
X.~Ros-Oton and J.~Torres-Latorre,
\textit{New boundary Harnack inequalities with right hand side},
Arch.\ Ration.\ Mech.\ Anal.\ \textbf{234} (2019), 1413--1444.

\bibitem{SavinYu2021}
O.~Savin and H.~Yu,
\textit{Regularity of the singular set in the fully nonlinear obstacle problem},
J.\ Math.\ Pures Appl.\ (9) \textbf{154} (2021), 1--22.

\bibitem{Signorini1959}
A.~Signorini,
\textit{Questioni di elasticit\`a non linearizzata e semilinearizzata},
Rend.\ Mat.\ e Appl.\ (5) \textbf{18} (1959), 95--139.

\bibitem{Taylor1978}
J.~E. Taylor,
\textit{Crystalline variational problems},
Bull.\ Amer.\ Math.\ Soc.\ \textbf{84} (1978), no.\ 4, 568--588.

\bibitem{Weiss1999}
G.~S. Weiss,
\textit{A homogeneity improvement approach to the obstacle problem},
Invent.\ Math.\ \textbf{138} (1999), no.\ 1, 23--50.

\bibitem{Wulff1901}
G.~Wulff,
\textit{Zur Frage der Geschwindigkeit des Wachstums und der Aufl\"{o}sung der
Krystallfl\"{a}chen},
Z.\ Kristallogr.\ \textbf{34} (1901), 449--530.

\end{thebibliography}
\end{document}